\documentclass[11pt,a4paper]{article}
\usepackage{epsfig}

\usepackage{hyperref}
\usepackage{amsmath,amsthm,amssymb,latexsym}
\usepackage{amsfonts}
\usepackage[american]{babel}
\usepackage{mathrsfs}

\newtheorem{theorem}{Theorem}[section]
\newtheorem{definition}[theorem]{Definition}

\newtheorem{rem}[theorem]{Remark}
\newtheorem{proposition}[theorem]{Proposition}
\newtheorem{ex}[theorem]{Example}

\newtheorem{claim}[theorem]{Claim}

\newcommand{\R}{{\Bbb R}}

\def\j1n{j=1,\dots,n}
\def\j1m{j=1,\dots,m}
\def\i1np1{\in +1}

\def\R{\mathbb{R}}

\def\rn{\mathbb{R}^n}

\def\i1np1{\in +1}

\def\R{\mathbb{R}}

\def\rn{\mathbb{R}^n}

\def\u1{u^{(1)}}

\def\h1{h^{(1)}}

\newcommand{\la}{\lambda}
\newcommand{\al}{\alpha}

\newcommand{\ga} {\gamma}
\newcommand{\Ga} {\Gamma}

\newcommand{\om}{\omega}

\newcommand{\RR}{{\mathbb R}}

\newcommand{\pa}{\partial}

\newcommand{\lims} {\lim\limits}

\newcommand{\mis}{\mbox{\rm mis }}

\begin{document}

\begin{title}{ On variational problems related to steepest descent curves
 and self dual convex sets on the sphere}
\end{title}

\author{Marco Longinetti\footnote{marco.longinetti@unifi.it, Dipartimento DiMaI,
Universit\`a degli Studi di Firenze, V.le Morgagni 67,  50134
Firenze - Italy, cell. 3204324446}\\
Paolo Manselli\footnote{paolo.manselli@unifi.it, 
Firenze - Italy, cell 3473126658}\\
Adriana Venturi\footnote{adriana.venturi@unifi.it, Dipartimento
GESAAF, Universit\`a degli Studi di Firenze, P.le delle Cascine 15,
50144 Firenze - Italy, cell 3204324449.}}
\date{}
\maketitle

\begin{small}{\bf Abstract. } Let $\mathcal{C}$ be the family of
compact convex subsets $S$
 of the hemisphere in $\rn$  with the pro\-perty that
$S$ contains its dual $S^*;$ let $u\in S^*$, and let $
\Phi(S,u)=\frac{2}{\omega_n}\int_{ S}\langle  \theta , u \rangle
\,\, d\sigma(\theta). $ The problem to study $ \inf \big\{\Phi(S,u),
 S \in \mathcal{C}, \, u\in S^* \big\} $ is considered. It is proved that the
minima of $ \Phi  $ are sets of constant width $ \pi/2 $ with $ u $
on their boundary. More can be said for $n=3$: the minimum set is a
Reuleaux triangle on the sphere. The previous  problem is related to
the one to find the maximal length of steepest descent curves for
quasi convex functions, satisfying suitable constraints. For $ n=2 $
let us refer to \cite{Manselli-Pucci}. Here  quite different results are obtained for $ n\geq 3$.

\end{small}



\section{Introduction}

Let $\Ga$ be the family of the rectifiable curves $\ga\subset \RR^n
\; (n \geq 2),$  satisfying a.e.:
   \begin{equation}\label{diffsep}
  \langle x'(s), x(s) -x(\sigma) \rangle  \geq 0 ,
   \quad \sigma \leq s ; \quad s,\sigma \in T
 \end{equation}
where   $x(s)$ is the curvilinear abscissa representation of $\ga,\, T
= [0,||\ga||], ||\ga||$ is the length of $\ga$.
 The curves $\ga$ satisfying  \eqref{diffsep} turn out
to be `` steepest descent curves'' for suitable quasi convex
functions. Further properties for these curves and related questions
were studied in \cite{Manselli-Pucci}, \cite{Daniilidis},
\cite{Daniilidis2} and also recently in \cite{Daniilidis3},
\cite{LMV}.

The mean width of a convex body (nonempty compact convex set) $K\subset
\rn$ will be denoted by $w(K)$. The convex hull of a set $ E $ will
be denoted by $co(E)$. Let $W$ be a positive number and $ \Ga_W $ be
the family of curves in $\Ga$ such that $w(co(\ga))\leq W$. In \S
\ref{pre} properties of these curves are recalled.

In the present work variational problems related to the family $\Ga_W$ are studied.

Let $\mathcal{C}$ be the family of compact convex subset $S$
 of the hemisphere in $\rn$ (in what follows, for simplicity they will be called sectors)
 with the property that
$S$ contains its dual $S^*;$ let $u\in S^*$, and let
$$
\Phi(S,u)=\frac{2}{\omega_n}\int_{ S}\langle  \theta , u \rangle
\,\, d\sigma(\theta).
$$
The first variational problem is to study
\begin{equation}\label{I}
 \mathcal{I} = \inf \big\{\Phi(S,u), \, S \in \mathcal{C}, \,
u\in S^* \big\}.
\end{equation}

 In \S  \ref{varprobcones} it is proved that $ \mathcal{I}>0,  $ and it is a
 minimum; a minimizing couple $ (S_0,u_0) $ satisfies $ S_0=S_0^*$
 (this is equivalent to the fact that
 $S$ has constant width  $\pi/2 $ on the sphere),  $ u_0 \in \partial S_0^*$. It is
 noticed that the case $ n=2 $ is trivial: the minimizing $ S_0 $
 is an arc of the unit circle of length $ \pi/2 $ and $ u_0 $ is  one of its extreme vectors.

 In \S  \ref{constwidthonS},  for  $ n=3,$ a sharper result is proved,
namely that $( S_0, u_0) $ is a minimizing couple for $ \Phi $ if and
only if $S_0$ is a Reuleaux triangle on the sphere
with $u_0$ one of its vertices; thus
$$ \min \big\{\Phi(S,u), \, S \in \mathcal{C}, \,
u\in S^* \big\} =  \Phi(S_0,u_0)=1/8.$$
 The proof
uses the   result of Blaschke-Lebesgue Theorem on $S^2$ on the
sectors of constant mean width
(\cite{Leichtweiss}). For $n > 3$  finding the  value of
$\mathcal{I}$ is an open problem.

The relation of previous problem with steepest descent curves was
hinted in \cite{Manselli-Pucci} and made more precise in \S
\ref{varprobconesfin}.

Let us define
$$\ga(s)=\{x(\sigma), 0 \leq \sigma \leq s \}$$
and
\begin{equation}\label{w(ga_s)}
 w_\ga(s)=w(co(\ga(s))).
\end{equation}
Let $N(x(s))$ be the normal cone to $co(\ga(s))$ at $x(s)$ and
$\widehat{N(x(s))}=N(x(s))\cap S^{n-1}$ is the related sector.

In  \cite{LMV} it is proved that for $ \ga
\in \Ga, $ a.e. $0 \leq s \leq ||\ga||$,
\begin{equation}\label{infdw/ds}
 \frac{d}{ds}w_\ga(s) =\Phi(\widehat{N(x(s))}, x'(s)).
\end{equation}
Thus:
\begin{equation}\label{ine}
 \inf_{\ga \in \Ga} \big \{ \frac{d}{ds}w_\ga(s) :  0\leq s \leq
 ||\ga|| \big \} \; \geq \;  \min \big\{\Phi(S,u), \, S \in \mathcal{C}, \,
u\in S^* \big\},
\end{equation}
then, $\Phi(\widehat{N(x(s))}, x'(s))$ has an infimum depending on $n$ only.

 In \cite{Manselli-Pucci}, for $n=2$, it was found a
$C^1$ curve $\ga_0$ for which  \eqref{ine}  (for every $s>0$) is an
equality.

Therefore it was natural to look for a similar result  for $ n=3.$
The results obtained are  different. In Theorem \ref{bad}
it is proved that
 $$\inf_{ \ga \in \Gamma_W } \frac{d}{ds}w_\ga(s) > \Phi(S_0, u_0),$$
(with $u_0$ one of the vertices of the Reuleaux triangle $ S_0$)
  holds a.e. in $[0, ||\ga||]$.

 However (Theorem \ref{bastard}) a curve  $\ga_0\in \Ga_W$ can be constructed so that there
 exist $ s_0 \in (0,||\ga_0||) $ and a sequence $s_k \rightarrow s_0$ with the property
 $$\lims_{k\to \infty} \frac{d}{ds}w_{\ga_0}(s_k)=\Phi(S_0,u_0)=\frac18.$$

 In \S \ref{varprobconesfin} an existence theorem is proved  for  the maximum of
 variational problems related to steepest descent curves; let
$$
\Psi(\ga,\al,  K)= \dfrac{||\ga||}{(w(K))^\al},\quad \ga\in\Ga,
\quad \ga \subset K, \quad 0 \leq \al \leq 1.
$$
In \cite{Manselli-Pucci}(see also \cite{Daniilidis3}, \cite{LMV}) it
was proved that, if $ w(K)$ is bounded a priori, then   $\Psi(\ga,
1, co(\ga))$ has an upper bound depending on $n$ only. Here is
proved that: if $K$ satisfies a constraint  implying that $w(K)$ is
bounded, then $\Psi(\ga, \al,  K)$ has a maximum.

\section{Preliminaries and definitions}\label{pre}
Let
$$
B(z,\rho)=\{x\in\rn\,:\,|x-z|<\rho\}\,,\quad\,
S^{n-1}=\partial B(0,1) \, \quad n\geq 2.
$$

A nonempty,  compact convex set $K$ of $\rn$ will be called  a {\em
convex body}.
$\pa K$ denotes
the boundary of $K$, $cl(K)$  is    the closure of $K$.

For  a convex body $K\subset \rn$, the {\em support function} is
defined by

\begin{equation*}
\label{equation1} H_K(x)=\sup_{y\in K}\langle x,y\rangle\,,\quad
x\in\mathbb{R}^n,
\end{equation*}
where $\langle \cdot, \cdot \rangle$ denotes the  scalar
product in $\mathbb{R}^n$. The restriction of $H_K$ to $S^{n-1}$ will be  denoted by
$h_K$.

 The width of a convex set $K$  in a
direction $\theta \in S^{n-1}$ is the distance between the two
hyperplanes  orthogonal to $\theta$ and  supporting $K$,  given by
$h_K(\theta)+ h_K(-\theta)$. The {\em mean width } $w(K)$ of $K$ is
the mean of this distance on $S^{n-1}$ with respect to the spherical
Lebesgue measure $\sigma$, i.e.
\begin{equation}\label{meanwidthdef}
w(K)= \dfrac{1}{\omega_n}\int_{S^{n-1}} \left( h_K(\theta)+ h_K(-\theta)\right)\,
d\sigma= \dfrac{2}{\omega_n} \int_{S^{n-1}}  h_K(\theta)\, d\sigma
\end{equation}
where $\omega_n=2\pi^{n/2}/\Gamma(n/2)$ is the  measure  of   $S^{n-1}$.

Let $K$ be a convex body and $q\in  K$; the {\em normal cone } at
$q$ to $K$ is the closed convex cone

\begin{equation}\label{normalcone} N_K(q)=\{x\in\rn: \langle x,y-q\rangle \le 0 \quad
\forall y \in K\}.
\end{equation}

\begin{definition}
Let $K$ be a convex body  and $p$ be a point not in K. A  simple
cap body $K^p$ is:
\begin{equation}\label{capbody}
K^p=\bigcup_{0\le \lambda \le 1}\{ \la K+(1-\la)p\}=co(K\cup\{p\}).
\end{equation}
\end{definition}

Cap bodies properties can be found in \cite{Bonnfen}; if $ p \in K,
$ let us define $K^p=K$. When $N$ is a cone let  $\widehat{N}$ be
the {\em sector} $N\cap S^{n-1}$ associated to $N$.
The differential properties of $p \to w(K^p)$  has  been investigated in $\RR^3$ in
 \cite[Satz VI]{Stoll} and in $ \RR^n $ in \cite[Theorem 2.4]{LMV}.
\begin{proposition}\label{main1}
$w(K^p)$ is a convex function of $p$ and for $p\not  \in \pa K$ it is differentiable  with
$$\nabla w(K^p)=\frac{2}{\omega_n}\int_{\widehat{N(p)}}\theta\, d\sigma ,$$
where $N(p)=N_{K^p}(p)$.
\end{proposition}

The {\em dual cone} $C^*$ of a  closed convex cone $C$ is
$$C^*=\{y\in \rn : \langle y,x \rangle \ge 0 \quad \forall x \in C\}.$$
The dual cone $C^*$ is a closed and convex cone.
A cone is said to
be {\em self-dual }if $C=C^*$. The nonnegative orthant of $\rn$,
$\rn_+=\{x=(x_1,\ldots, x_n): x_i\ge 0, i=1, \ldots, n \}$ and the
circular cone $C_+=\{x=(x_1,\ldots, x_n): x_n^2=\sum_1^{n-1}
x_i^2\}$ are self-dual. For $n=2$ the self-dual cones are  rotations of
the quadrants only; for $n \geq 3 $ the self-dual cones are related
to the convex sets on the sphere of constant width $\pi/2$.

The opening of a circular cone will be  the amplitude of the acute angle between the
axis and a generator half line. If $C$ is a circular cone of opening $\alpha$ then $C^*$
is a circular cone of opening $\pi/2-\alpha$.
The  {\em tangent cone}, or  support cone, of a convex body  K at a
point $q \in \pa K$  is given by
$$ T_K(q)=cl \{\bigcup_{y\in K} \{s( y-q): s \geq 0  \}\}.$$
It is well known that:
\begin{equation}\label{N=-T*}
(N_K(q))^*=-T_K(q).
\end{equation}

An extreme ray of a convex cone $C$, is a ray in $C$ that cannot be
expressed as a sum of two other rays in $C$. An extreme point of a
sector can be defined in a similar way.

Let us recall that the Hausdorff distance  between two convex bodies $A$ and $B$ can be
written as
$$dist(A,B)= \max_{\theta \in S^{n-1}}|h_A(\theta)-h_B(\theta)|$$
(see \cite[Theorem 1.8.11]{Schn}).


\begin{definition}\label{defgax}
Let $x_0=x(s_0)\in \ga$, let us denote
$$\ga_{x_0}=\ga(s_0):=\{x(s)\in \ga, \,  0\leq s \leq s_0\}.$$
 \end{definition}

The principal properties of the curves $\ga \in \Ga$  follow,
 see  \cite{Manselli-Pucci}, \cite[\S 4]{LMV}.
 \begin{proposition}\label{fondproprsdc} Let $x(t), t\in [0,1],$ be a  parametrization of a curve  $\ga$. Then, $\ga\in \Ga$ iff 
 \begin{equation}
  |x(t')-x(t'')| \leq |x(t')-x(t''')| \quad 0\leq t'\leq t''\leq t'''\leq 1.
 \end{equation}

 \end{proposition}

\begin{proposition}\label{proprsdc} If $\ga\in \Ga$ and $x\in\ga$,
  then for any $p,q \in co(\ga_x)\setminus \{x\}$
\begin{equation}\label{angcond}
\langle p-x,q-x\rangle >0,
\end{equation}
and any two half lines from $x\in \ga$ in the tangent cone at $co(\ga_x)$ are the
sides of an  angle  less than or equal to $\pi/2$; moreover
\begin{equation}\label{NpsubsetTp}
N_{co(\ga_x)}(x)\supseteq (N_{co(\ga_x)}(x))^* = -T_{co(\ga_x)}(x).
\end{equation}
\end{proposition}

Let  $\ga \in \Ga$. The path
$\ga$ has a one-to-one continuous parametrization $x(w)$,  inverse of
\begin{equation}\label{pargaw}
w(x):=w(co(\ga_x))\in [0, w(co(\ga))].
\end{equation}

\begin{proposition}\label{selfsecrettifiable} Let $\ga\in \Gamma$.
 Then   $\ga, $ parameterized by the mean width
function \eqref{pargaw} is Lipschitz continuous, and a.e.
\begin{equation}\label{dx/dw<=c(n)}
 |\dfrac{dx}{dw}| \leq c^{(1)}_n,
\end{equation}
where $ c^{(1)}_n $ is  a constant  depending only on the dimension
$n$. In particular $$ ||\ga|| \leq c^{(1)}_n w(co(\ga)).$$ Moreover
$ c^{(1)}_2 = \pi $ (best possible constant) and
 $$
  c^{(1)}_n \leq (n-1) \cdot n^{n/2} \frac{\omega_n}{\omega_{n-1}} .$$
\end{proposition}

\begin{proposition}\label{corIBDpath} Let $\ga\in \Ga$ and let
 $N(x(s))=N_{co(\ga(s))}(x(s))$; for  almost every $s\in (0, ||\ga||),$
\begin{equation}\label{inNco(ga(t))}
{x'}(s)\in N^*({x(s)})\subset N({x(s)})
\end{equation}
holds and
\begin{equation}\label{MP}
 \frac{d}{ds}w_\ga(s)= \frac{2}{\omega_n}\int_{\widehat{N(x(s))}} \langle
\theta,{x'}(s) \rangle d \sigma.
\end{equation}

\end{proposition}

From Proposition \ref{fondproprsdc} it follows
\begin{theorem}\label{seqpath} Let $\{ \ga_m \}$ be a sequence of
steepest descent curves and $\ga_m$
parameterized as $ x_m( t ),\, t \in [0,1]. $ If the sequence $
\{x_m(t)\} $ converges uniformly to $ x_0(t), \, t \in [0,1], $ then
$x_0(t)$ is a parametrization of a steepest descent curve.
\end{theorem}

\section{A variational problem on  sectors}\label{varprobcones}
Formula \eqref{MP} and the results presented in \S
\ref{varprobconesfin} of the present paper, suggest to study some
special functionals on  sectors (i.e. closed, convex subsets) of the
sphere $S^{n-1}$ which are investigated here.

Let $S$ be a  sector  in $S^{n-1}$, $n\geq 2$ and $u$ a unit vector
in $S$. Let us define
$$\Phi(S,u)=\frac{2}{\omega_n}\int_{ S}\langle  \theta , u \rangle^+ \,\,
 d\sigma(\theta),$$
where $f^+=\max\{f,0\}$ is the positive part of $f$. The {\em dual
body} $S^*$ of a convex sector $S$ is
$$S^*=\{y\in S^{n-1} : \langle y,x \rangle \ge 0 \quad \forall x \in S\}.$$
$S$ is a self dual sector when $S^*=S$.

For our scopes the functional $\Phi $ will be restricted to a
special class of sectors: the class $\mathcal{C}$ of  sectors $S$
such that $S\supseteq S^*$. The following variational problem:
\begin{equation}\label{probminPHI}
 \inf \{\Phi(S,u), u\in S^*, S\in \mathcal{C}\}
\end{equation}
will be considered.

Since $u\in S^*$, the function $\langle  \theta , u \rangle$ is non
negative in $S$. Our result is
\begin{theorem}\label{minfunoncones}  The functional
\begin{equation}\label{defPHI}
 \Phi(S,u)=\frac{2}{\omega_n}\int_{ S}\langle  \theta ,
 u \rangle \,\, d\sigma(\theta), \quad u\in S^*
\end{equation}
has  a positive minimum   on the class $\mathcal{C}$. The minimum is
attained at a self-dual sector $S $  and at a vector $v$ such that
$v\in \pa S $.
\end{theorem}
\begin{proof} The class $\mathcal{C}$ is obviously compact with respect
to the Hausdorff distance and  $\Phi$ is continuous with respect to
this metric. Moreover the subset of the self dual sectors is a compact set too. In \cite{Barker} it is proved that, if $C$ is a convex
cone containing its dual, then there exists a self-dual cone $K$
contained in $C$. Similar results hold for sectors. Therefore if $ S
\in \mathcal{C},$ then there exists a  self dual sector $ S_0\subseteq S,
$ $S_0 = S_0^*\supseteq S^*;$ moreover, if $u\in S^*,$ then $u\in
S_0^*$ and
$$\Phi(S,u) \geq  \Phi(S_0,u) \geq \min_{S=S^*,\, v\in S^*} \Phi(S,v).$$
To conclude the proof let us observe that: i) $\Phi$ depends
linearly on $u$; ii) its gradient with respect to $u$ does not
vanish in $S$. Therefore if $\Phi$  has minimum at $(S,v)$, then
$v\in \pa S.$ To conclude that this minimum is positive,
\cite[Theorem VI]{Manselli-Pucci} (coming  from Theorem 1 of
\cite{Santalo'}) can be used. In our notations,  if $S$ is a
sector in $S^{n-1}$, $n\geq 2$, then
$$\frac{2}{\omega_n}\int_{ S} \langle \theta , u \rangle \,\,
 d\sigma(\theta)\geq \frac{\om_{n-1}}{\om_n(n-1)}(\dfrac{1}{\sqrt{n}})^n$$
holds.
\end{proof}
\begin{rem}
 In \cite{Bezdek}, \cite{Dekster} it has been noticed that the self dual
 sectors are sectors with constant width $\pi/2$ and conversely.
\end{rem}

\begin{rem} \label{rem3.3} If $ n=2 ,\,  \mathcal{C} $ is the family of arcs $\; S =
\Big \{ \theta = (\cos \vartheta , \sin \vartheta  \big) , \; \vartheta_1 \leq \vartheta \leq \vartheta_2 \Big \} $
of the unit circle in $ \RR^2 $ with $\vartheta_2-\vartheta_1 \in
[\pi/2, \pi] $ and $ u = (\cos \zeta , \sin \zeta
 \big) $ is a vector with end point in $ S. $ Moreover
$$
\Phi(S,u)= \frac{1}{\pi} \int_{\vartheta_1}^{\vartheta_2}
\cos(\zeta - \vartheta ) d\vartheta.
$$
The minimum of $ \Phi $ is trivial: the minimizing $ S_0 $
 is an arc of  length $ \pi/2 $ and $ u_0 $ one of
 unit vectors joining the origin with one of the end points of the
 arc. Thus
 $$
\min \{\Phi(S,u), u\in S^*, S\in \mathcal{C}\}= \Phi(S_0,u_0 )=
 \frac{1}{\pi} \int_{0}^{\pi/2} \cos(\vartheta ) d\vartheta = \frac{1}{\pi}.
 $$
\end{rem}
\section{The minimum of the functional $\Phi$ for  n=3}\label{constwidthonS}
First let us introduce some notations and definitions.

Let $S^+=\{x_1^2+x_2^2+x_3^2=1, x_1 \geq 0\}$ be  the upper hemisphere in $\R^3$ centered at O=(0,0,0).
Let $p\in S^+, q \in S^+$, $p\neq -q$.
Let us denotes $\widehat{\overline{pq}}$ the shorter arc of the great circle of $S^+$ with endpoints $p$ and $q$.

$S^+, $ with the distance
$$|\widehat{\overline{pq}}|:=\arccos \langle p, q \rangle \mbox{\quad for\, } p\neq -q ,
\quad  |\widehat{\overline{pq}}|:=\pi \mbox{\quad for\, } p=-q,$$
is a metric space. For every $p\in S^+$, let
$$[p]:=\{x\in \R^3: x=\la p, \la \geq 0\}$$
 the  half line from O through $p$.
$\mathcal{K}$ will be called {\em convex in } $S^+$  if and only if
it is a sector. Let us notice that if $\mathcal{K}$ is convex in
$S^+$ and $p \in \mathcal{K}, q \in \mathcal{K}$ then
$\widehat{\overline{pq}} \subset \mathcal{K}$.


A sector $\mathcal{K}$ will be called of constant width $w\, (0< w
\leq \pi/2)$ if for every $\theta_0\in \pa \mathcal{K}$ there exists
$\theta_0'\in \pa \mathcal{K}$ satisfying
$|\widehat{\overline{\theta_0\theta_0'}}|= w$, and
$|\widehat{\overline{\theta_0\theta_1}}|\leq  w$ for all $\theta_1
\in \mathcal{K}.$ Examples of sectors $\mathcal{K}$ on  $S^+$ of
constant width  are obtained by intersection of $S^+$ with any
circular cone $K$ contained  in $x_1 \geq 0$. The width of
$\mathcal{K}$ is twice of the opening of $K$.

The intersection of {\em the positive orthant}  $\{x=(x_1,x_2,x_3): x_i \geq 0 , i=1,2,3\}$
 with $S^+$ is a convex set of constant
width $\pi/2$. Let us  call {\em orthant } any set contained in
$x_1\geq 0 $ which is a rotation of the positive orthant. Let us
call {\em $S^+$orthant} the  intersection of $S^+$ with an orthant; in \cite{Bezdek} an $S^+ $ orthant is also called a
Reuleaux triangle on the sphere.

The Blaschke-Lebesgue Theorem on $S^2$ (\cite{Leichtweiss}) is needed to prove the following
\begin{theorem}\label{minorthant} Let $\mathcal{K}$ a convex set in $S^+$ of constant width  $\pi/2$, and $u\in \pa\mathcal{K}.$
Let $\mathcal{O}$ an  $S^+$orthant  having $u$ as an extreme point. Then
\begin{equation}\label{intminorthant}
\int_{\mathcal{K}}\langle u, \theta \rangle d \sigma(\theta) \geq \int_{\mathcal{O}}\langle u, \theta \rangle d \sigma(\theta
)=\frac{\pi}{4}.
\end{equation}
The equality holds if and only if $K$ is an $S^+$orthant.
\end{theorem}
 From now on it will be assumed that a convex subset of $S^+$ will have
the boundary counterclockwise oriented.
\begin{proof}Let us
assume that $u=(1,0,0)$ and that $\mathcal{K}$ is not an orthant. Let
$U=\{u'\in \pa \mathcal{K}: |\widehat{\overline{uu'}}|=\pi/2\}$. Then
$U$ is either a point $u'$ or a closed arc of the equator $\{x_1,x_2,x_3): x_1=0, x_2^2+x_3^2=1\}$
of the form $\widehat{\overline{u'_1u'_2}}$, with  $|\widehat{\overline{u'_1u'2}}|< \pi/2$.

There are infinite $S^+$ orthants $\mathcal{O}$ containing both $u$ and
$U$. Each of them is bounded by three arcs of great circles,
$\widehat{\overline{uu_1}}, \widehat{\overline{u_1u_2}} \supset U,
\widehat{\overline{u_2u}}$. Let us call $\mathcal{O}(u_1,u,u_2)$
such $S^+$ orthant.  If $U \subset \widehat{\overline{u_1u_2}}$ (extremes
excluded) then $u_1 \not \in \mathcal{K}$ and
$\widehat{\overline{uu_1}} $ will be decomposed in two arcs
$\widehat{\overline{u\theta_1}}\subset \mathcal{K}$ and
$\widehat{\overline{\theta_1u_1}}\cap \mathcal{K}=\{\theta_1\} $;
moreover $u_2 \not \in \mathcal{K}$ and $\widehat{\overline{u_2u}}$
will be decomposed in two arcs
$\widehat{\overline{\theta_2u}}\subset \mathcal{K}$ and
$\widehat{\overline{u_2\theta_2}} \cap
\mathcal{K}=\{\theta_2\}$ (see Figure 1). If $u_1=u_1'$   then
$\widehat{\overline{u_1u}}\subset \mathcal{K} $. If $u'_2=u_2$  then
$\widehat{\overline{u_2u}}\subset \mathcal{K} $. If $u_1$ moves
continuously towards $U$, $\theta_1$ moves continuously toward $U$
too; similarly behaviour for $u_2$. Therefore there exists a choice
of $u_1,u_2$ so that the corresponding $\theta_1,\theta_2$ (called
$\tau_1,\tau_2$  from now on) satisfy
\begin{figure}[htb]
\epsfig{file=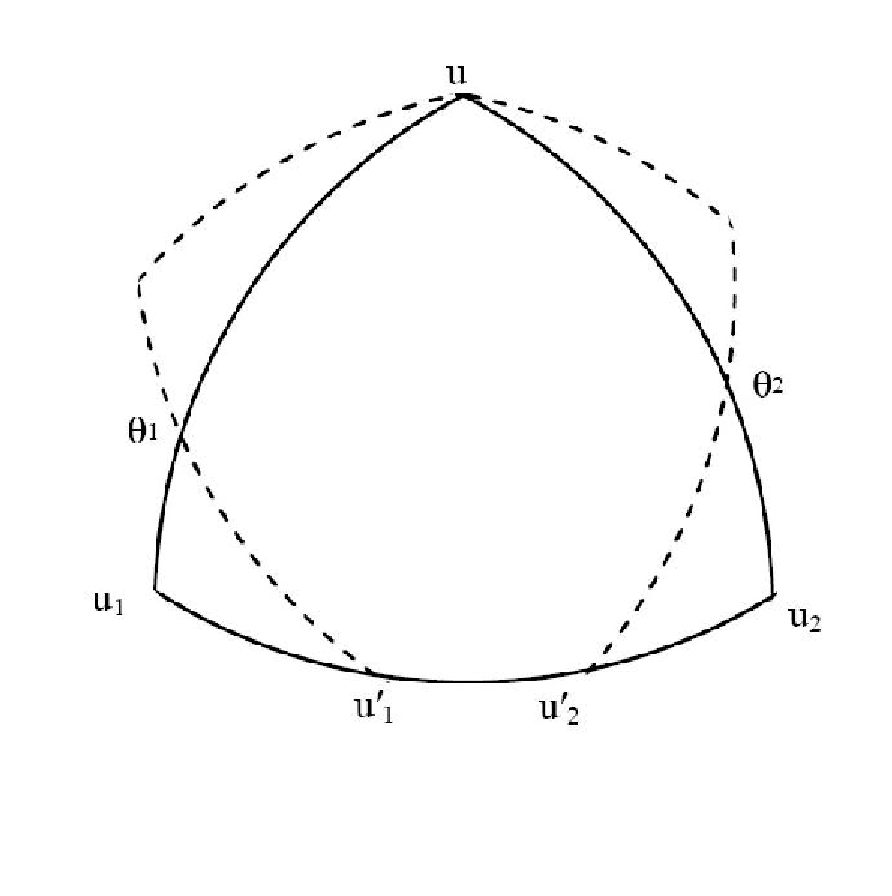, width=10cm} \caption{ Sets of constant width  $\pi/2 $.}
\end{figure}

\begin{equation}\label{uta1=utau2}
|\widehat{\overline{u\tau_1}}|=|\widehat{\overline{u\tau_2}}|.
\end{equation}

Let us consider $\mathcal{K}$ and the $S^+$ orthant $ \mathcal{\tilde{O}}=\mathcal{O}(u,u_1,u_2)$ satisfying \eqref{uta1=utau2}.
Let us consider also
the spherical sector $\mathfrak{C}=\{\theta \in S^+: |\widehat{\overline{\theta u}}|\leq |\widehat{\overline{\tau_1 u}}|\}.$
Notice that $\mathcal{K}\cap \mathfrak{C} $, $\mathcal{\tilde{O}}\cap \mathfrak{C}$ are  sectors and
$\pa(\mathcal{\tilde{O}}\cap \mathfrak{C})\subset \pa(\mathcal{K}\cap \mathfrak{C})$, thus
$$\mathcal{K}\cap \mathfrak{C} \supseteq \mathcal{\tilde{O}}\cap \mathfrak{C}, \quad \mathcal{K}\setminus \mathfrak{C}
\subseteq \mathcal{\tilde{O}}\setminus \mathfrak{C}.$$
Let  us compute:
$$J:=\int_{\mathcal{K}}\langle u, \theta \rangle d \sigma(\theta)=\int_{\mathcal{K}\cap\mathfrak{C}}\cos |\widehat{\overline{\theta u}}|
d \sigma(\theta)+\int_{\mathcal{K}\setminus\mathfrak{C}}\cos |\widehat{\overline{\theta u}}| d \sigma(\theta)=$$
$$=\int_{\mathcal{\tilde{O}}\cap\mathfrak{C}}\cos |\widehat{\overline{\theta u}}| d \sigma(\theta)+
\int_{(\mathcal{K}\setminus\mathcal{\tilde{O}})\cap\mathfrak{C}}\cos |\widehat{\overline{\theta u}}|
d \sigma(\theta)+\int_{\mathcal{\tilde{O}}\setminus\mathfrak{C}}\cos |\widehat{\overline{\theta u}}|
d \sigma(\theta)-\int_{(\mathcal{\tilde{O}}\setminus\mathcal{K})\setminus \mathfrak{C}}\cos |\widehat{\overline{\theta u}}|
d \sigma(\theta)=$$
(adding first and third integral)
$$=\int_{\mathcal{\tilde{O}}}\langle u, \theta \rangle d \sigma(\theta)+\int_{(\mathcal{K}\setminus\mathcal{\tilde{O}})
\cap\mathfrak{C}}\cos |\widehat{\overline{\theta u}}| d \sigma(\theta)-\int_{(\mathcal{\tilde{O}}\setminus\mathcal{K})\setminus
\mathfrak{C}}\cos |\widehat{\overline{\theta u}}| d \sigma(\theta).$$
Using the fact that
$$\left.\cos |\widehat{\overline{\theta u}}| \right|_{\mathfrak{C}}\geq \cos |\widehat{\overline{\tau_1 u}}| \geq
\left.\cos |\widehat{\overline{\theta u}}|\right|_{(\mathcal{\tilde{O}}\setminus\mathcal{K})\setminus \mathfrak{C}}, $$
we have
\begin{equation}\label{jineq}
J \geq \int_{\mathcal{\tilde{O}}}\langle u, \theta \rangle d \sigma(\theta)+\cos |\widehat{\overline{\tau_1 u}}|\left[
\mis \left( (\mathcal{K}\setminus\mathcal{\tilde{O}})\cap\mathfrak{C}\right)-\mis
\left((\mathcal{\tilde{O}}\setminus\mathcal{K})\setminus \mathfrak{C}\right) \right].
\end{equation}

Since
$$\mis \left( (\mathcal{K}\setminus\mathcal{\tilde{O}})\cap\mathfrak{C}\right)=\mis
(\mathcal{K}\cap\mathfrak{C})-\mis (\mathcal{\tilde{O}}\cap\mathfrak{C})$$
and
$$ \mis (\mathcal{\tilde{O}}\setminus\mathcal{K})\setminus \mathfrak{C}=
 \mis (\mathcal{\tilde{O}}\setminus\mathfrak{C})-\mis (\mathcal{K}\setminus\mathfrak{C}),$$
subtracting the last two equalities we get
$$ \mis \left( (\mathcal{K}\setminus\mathcal{\tilde{O}})\cap\mathfrak{C}\right)-\mis
\left((\mathcal{\tilde{O}}\setminus\mathcal{K})\setminus
\mathfrak{C}\right)=\mis \mathcal{K}- \mis \mathcal{\tilde{O}}.$$
 From the Blaschke-Lebesgue theorem on the sphere \cite{Leichtweiss}, $\mis \mathcal{K}- \mis \mathcal{\tilde{O}}$
is  positive, so the last one  in \eqref{jineq} is
positive too. Since \cite[Theorem 5.3, Remark 5.10]{Leichtweiss}
equality holds in (\ref{intminorthant}) if and only if $K$ is an $S^+$ orthant.\end{proof}

From previous theorem and Theorem \ref{minfunoncones} for  $n=3$ the variational problem  \eqref{probminPHI}
is now solved. 
\begin{theorem}\label{n=3equality} If $\, n = 3, $
 $$
 \min \{\Phi(S,u), u\in S^*, S\in \mathcal{C}\} \geq   \frac{1}{8}.
 $$
 The equality holds if and only if $S$ is a $S^+$ orthant and $u$ is
 one of its extreme points.
\end{theorem}

\section{Variational problems for steepest descent curves and related
 problems}\label{varprobconesfin}
 In this paragraph, using the results of previous ones, a priori lower  bounds
 for the derivative of the mean width of the convex envelope of steepest descent curves are studied; an existence theorem for
 related functionals is also proved.
\begin{theorem}\label{minderws}Let  $\ga \subset \rn , n\geq 2, \ga\in \Ga$.
The derivative of the mean width of the convex envelope of $  \ga $
is
 uniformly bounded from below; more precisely
\begin{equation}\label{n=3inequality}
 \dfrac{d}{ds}w_{\ga}(s)\geq \min \{\Phi(S,u), u\in  S^*, S\in \mathcal{C}\}.
\end{equation}

\end{theorem}
\begin{proof} Since $x'(s) \in -T_{co(\ga(s))}(x(s))$,
from equality \eqref{MP}, inclusion \eqref{NpsubsetTp} and
 Theorem \ref{minfunoncones}, the inequality  \eqref{n=3inequality}
follows.
\end{proof}

 In \cite{Manselli-Pucci} the following example was noticed.

 \begin{ex}\label{ex1}  Let $ \om \in \RR $ be the unique solution to the
equation $ \om = e^{-\frac{3 \pi}{2} \om} \; \sim \;0.2744, \; \mu >
0, $ $ 0~< s \leq  \mu; $ let us define
\begin{equation}\label{eta}
\eta(s):=-\frac{1}{\om} \log \big(\frac{ \om \, s }{\sqrt{1 +
\om^2}}\big);
\end{equation}
let $ \ga^{(\mu)} $ be the logarithmic
spiral in $ \RR^2, $ with parametric equations:
\begin{equation}
x_1(s) =\frac{\om \, s}{\sqrt{1 + \om^2}} \cos \eta(s) ,\quad
x_2(s) =\frac{\om \, s}{\sqrt{1 + \om^2}} \sin \eta(s).
\end{equation}
It can be shown that  $ \ga^{(\mu)} $ is a steepest descent curve
and $ s $ is its curvilinear abscissa; the convex body $
co(\ga^{(\mu)}(\mu)) $ has perimeter  $ \mu = \pi \; w(
co(\ga^{(\mu)}(\mu)). $
\end{ex}

\begin{rem} For $n=2$ the minimum of the functional in Theorem \ref{minfunoncones} is $1/\pi$ and there is
a steepest descent curve  $\ga^{\star}$ such that $\dfrac{d}{ds}w_{\ga^{\star}}(s)=1/\pi$ at each point.
\end{rem}
\begin{proof} It is an easy consequence of Remark \ref{rem3.3}.
 The remaining part is a consequence of Example \ref{ex1}, for every  $\mu >0$.
\end{proof}

Let us point out that, for $n=3,$  from Theorem \ref{n=3equality}
the value of
 the  right hand side in \eqref{n=3inequality}
is $\frac{1}{8}$ and it is obtained  only when $S$ is an  orthant, but
the equality case in \eqref{n=3inequality} is not attained  in three
dimensions as the following theorem shows.

\begin{theorem}\label{bad} Let $S_0 $ be an orthant of $ \RR^3 $ intersected with a semisphere,
$ u_0 $ one of its extreme points, $ \ga \in \Ga_W, s \in [0, ||\ga||] .$
The inequality
$$
\inf_{\ga \in \Ga_W} \frac{d w(co(\ga(s)))}{ds} > \Phi(S_0,u_0) =
\min \{\Phi(S,u), u\in  S^*, S\in \mathcal{C}\}
$$
holds  a.e. in $ [0, ||\ga||] .$
\end{theorem}
\begin{proof}
The proof is an immediate consequence of the following claim.
\begin{claim}\label{tangentorthant} For $n=3,$
for every curve $\ga \in \Ga, $  the tangent cone to $
co(\ga(s)) $ cannot be an orthant at any point where $x'(s)$ exists.
\end{claim}

Let us argue by contradiction. Let us assume that in a point $ x(s) $
of the curve $\ga$:
$x'(s)$ exists and
 the tangent cone to $co(\ga(x(s)))$ is an orthant, generated by
three unit vectors $v_1,v_2,v_3$ orthogonal to each other.
 As the
tangent cone has an edge for each $ v_j, j=1,2,3, $ by definition of
tangent cone, there are  two possible cases (for each $ j=1,2,3 $):
(i) there exists $ x(s^j)\in \ga, s^j<s $ satisfying
$$
\frac{x(s^j)-x(s)}{|x(s^j)-x(s)|} = v_j;
$$
(ii) there  exists a sequence $x(s_{k}^j) \in \ga $ with $s_{k}^j< s$
and $\lims_{k\to \infty}x(s_{k}^j)= x(s)   $ satisfying:
$$
\lims_{k\to \infty} \frac{x(s^j_k)-x(s)}{|x(s^j_k)-x(s)|} = v_j.
$$

There are two possibilities.
\begin{enumerate}
 \item case (i) for at least two edges; then the triangle joining $x(s)$ with two of the $x(s^j)$ is rectangle in $x(s)$
and this is impossible by   \eqref{angcond};
\item case (ii) for at least two edges; then $x'(s) $ cannot exist.
\end{enumerate}
\hfill\end{proof}

Let us notice, however the following example.
\begin{ex}\label{bastard}
Let $ n=3,\; W>0 $ and $ S_0, u_0 $ as defined in the previous
theorem. There exists $ \widehat{\ga} \in \Ga_W, $ $ \widehat{s} \in (0,
||\widehat{\ga} ||), $ and a sequence $ 0 < s_m \searrow \widehat{s} $ so
that
$$
 \lim_{m \rightarrow \infty} \frac{dw(co(\widehat{\ga}(s_m)))}{ds}= \Phi(S_0,u_0) = \frac{1}{8}.
$$
\end{ex}
\begin{proof}
 In Example \ref{ex1}
let us choose $\mu =\frac{W}2\frac{\sqrt{1+\om^2}}{\om}.$
Let  $ \widehat{\ga}$ the curve in $ \RR^3 $ in curvilinear abscissa $ s
$ defined as
$$
x(s) = \left(\begin{array}{c} 0 \\0 \\ s-W/2  \end{array} \right)
\qquad \qquad 0 \leq s \leq W/2,
$$
$$
 \qquad x(s) = \left(\begin{array}{c} \frac{\om \, \sigma}
 {\sqrt{1 + \om^2}} \cos \eta(\sigma) \\
\frac{\om \, \sigma}{\sqrt{1 + \om^2}} \sin \eta(\sigma) \\ 0
\end{array} \right) \quad \sigma = s-W/2, \; W/2 <s \leq \frac{W}2(\frac{\sqrt{1+\om^2}}{\om}+1)
$$
($ \eta $ defined in (\ref{eta})). The  curve  $ \widehat{\ga}$ is the
union of the oriented segment joining $(0,0, -W/2 )^t $ to $(0,0, 0)^t,$ 
followed by the logarithmic spiral $\ga^{(\mu)}$
(defined in Example \ref{ex1}) on the plane $x_3 =0.$ 
 The ball  centered in $(0,0,0)^t$ and radius $W/2$ contains $co(\widehat{\ga})$, thus $ \widehat{\ga}\in \Ga_W. $
Let  $\widehat{s} = W/2$ and $ s_m \searrow \widehat{s};
$ the tangent cone $ T_m $ to co($\widehat{\ga}(s_m ))$ is a trihedron
with vertex in $ x(s_m), $ two orthogonal edges on the $ x_1x_2$ plane,
containing $(0, 0 ,0)^t;$ the third edge is the half line starting in
$ x(s_m) $ and containing $( 0, 0, -W/2)^t $. As $ x(s_m) $ is on the $ x_1x_2$ plane and
tends to the origin, then $ T_m, $  as $ m \rightarrow \infty, $
tends to an orthant; thus $ \widehat{N(x(s_m))} $ tends to an
orthant too. By  \eqref{MP} and Theorem \ref{n=3equality},
the thesis follows.
\end{proof}

\vspace{1cm}

Let us consider now another related variational problem. A
consequence of \cite[theorem VII]{Manselli-Pucci} and its
generalization to the class of steepest descent curves (Proposition
\ref{selfsecrettifiable}), is that the arc length  of a steepest
descent curve $\ga$ is bounded by $w(co(\ga))$.

 Here a simple but sharper version of previous result is proved.

 Let $W>0$, $ \mathfrak{H}_W $ be a class of convex bodies contained  in $ \RR^n $ satisfying the
 conditions: (i) if $ K \in \mathfrak{H}_W $ and  $\ga \in \Ga$ is
 contained in $K $ then $ co(\ga) \in \mathfrak{H}_W; $ (ii)
$ \sup_{_{K \in \mathfrak{H}_W }}  w(K) \leq W;$ (iii) the class is closed with respect to the
 Hausdorff metric.

Examples of  classes with the  above properties are: \\
(a) $
\mathfrak{H}^{(1)}_W: $ the family of all convex bodies $ K $ of $
\RR^n $ satisfying $ w(K) \leq W; $
\\ (b) $ \mathfrak{H}^{(2)}_W: $
the family of all convex bodies $ K $ of $ \RR^2 $ contained in the
circle (centered in the origin) of radius $ W/2.$

Let $ \tilde{\Ga} \subset \Ga, $ be a family of steepest descent curves,
closed for the uniform convergence. As an example $ \tilde{\Ga} $ could be
the family of the curves of $ \Ga $ with the same starting point.

Let $ W>0, $ $ \mathfrak{H}_W $ be a class of convex bodies $ K $ in
$\rn$, satisfying (i), (ii), $\ga \in \tilde{\Ga} \subset \Ga,$ $\ga
\subset K . $

 Let us  consider a variational problem for  the  functional
\begin{equation}\label{probvargaK}
\Psi(\ga,\al, K)= \dfrac{||\ga||}{(w(K))^\al}, \quad  \ga \in
\tilde{\Ga}, \quad \ga \subset K \in  \mathfrak{H}_W, \quad 0 \leq
\al \leq 1.
\end{equation}
\begin{theorem}\label{maxprobvargaK} Let $ 0 \leq
\al \leq 1, W >0.$  Given  $\tilde{\Ga}$ ( closed with respect to the $||\cdot ||$ convergence), $\mathfrak{H}_W$,  there exist a curve $\ga_0\in \tilde{\Ga}$ and  a
 convex body ${K_0}=co(\ga_0) \subset \RR^n,$ $ K_0 \in
 \mathfrak{H}_W, $
 such that
$$\Psi({\ga_0},\al,{K_0})=\max \{\Psi(\ga,\al, K),\ga \in \tilde{\Ga}, \ga \subset K \in \mathfrak{H}_W \}.$$
\end{theorem}
\begin{proof}

 Let $ \ga $ and $ K $ satisfying: $ \ga \in \tilde{\Ga}, \ga \subset K \in \mathfrak{H}_W   $. As $
w(co(\ga)) \leq w(K) \leq W,$ Proposition \ref{selfsecrettifiable}
implies that
$$||\ga||\leq c^{(1)}_n W,$$
thus
$$
 \Psi(\ga,\al, K) =  \frac {||\ga||}{(w(K))^\al } \leq \frac {||\ga||}{(w(co(\ga)))^\al }
  = \Psi(\ga, \al, co(\ga))\leq
c^{(1)}_n W^{1-\al}.
$$
Thus the functional $ \Psi $ is bounded above in $ \mathfrak{H}_W. $
 Let $ \{ \ga_m \} $ an extremum sequence
 $$
\Psi(\ga_m, \al, co(\ga_m)) \nearrow \sup \{ \Psi(\ga, \al, K) , \ga
\in \tilde{\Ga}, \ga \subset K \in \mathfrak{H}_W ) \}.
 $$

Let us parameterize $  \ga_m  $ as
$$
\ \quad  x_m(t^{(m)}),   \quad t^{(m)}:=\frac{s}{|| \ga_m  ||}
\in [0,1], \quad 0 \leq  s \leq || \ga_m   ||, \quad s \mbox{ arc
length abscissa to } \ga_m.
$$
The functions  $ \{x_m \} $ are Lipschitz functions in $ [0,1], $
satisfying a.e. $  ||x'_m|| \leq  || \ga_m   || \leq c^{(1)}_n W. $
Thus, by Ascoli Arzel\'{a}
theorem, a subsequence of $ \{x_m \} $ uniformly converges to a
function $ x_0 $ parametric representation of a steepest descent
curve $ \ga_0 $ (Proposition \ref{seqpath}). As the functional $
\Psi $ is continuous in its arguments, by the assumptions on the class $\tilde{\Ga}$ and $\mathfrak{H}_W$ the thesis follows.
\end{proof}
The previous theorem  merely states the existence of a maximizing
steepest descent curve $ \ga_0 $  for the functional \eqref{probvargaK} in the
class $ \tilde{\Ga}, \mathfrak{H}_W. $ Find explicitly $ \ga_0 $  is
another problem.

 In \cite{Manselli-Pucci}, for $ n=2 , \al =1 $ and $ \mathfrak{H}_W
=\mathfrak{H}^{(1)}_W, $ $ \tilde{\Ga}=\Ga, $ a maximizing curve (the
logarithmic spiral in  Example \ref{ex1}) was found.

The results of Theorem \ref{bad} imply that the techniques used in
\cite{Manselli-Pucci} cannot be used for  $n=3. $ The problem to
find a maximizing steepest descent curve $ \ga_0 $  for the
functional \eqref{probvargaK} in the class $ \Ga,
\mathfrak{H}^{(1)}_W $ remains open.

Furthermore,  for $ n=2,  \al =0,  \mathfrak{H}_W
=\mathfrak{H}^{(2)}_W, $ $ \tilde{\Ga} $ the family of curves  of $
\Ga $ starting in the center of the circle, it was noticed in
\cite{Manselli-Pucci} that there is a steepest descent curve
$\tilde{\ga}: \tilde{x}(s), 0\leq s \leq 1+\pi$ where:
$$
\tilde{x}=(s,0) \quad 0\leq s< 1, \quad
\tilde{x}=(\cos(s-1),\sin(s-1)) ,
 1 \leq s \leq 1+\pi,
 $$
longer than the spiral $\ga^{(\mu)}$ quoted above. In \cite{IKL}
it has been computed the supremun (not maximum) of $||\ga||, \ga \in \mathfrak{H}^{(2)}_W$
and an extremal sequence is explicited constructed.

\end{document}